\documentclass{amsart}

\usepackage{amssymb,amsmath,graphicx,dsfont,color,url}

\newtheorem{theorem}{Theorem}[section]
\newtheorem{corollary}[theorem]{Corollary}
\newtheorem{lemma}[theorem]{Lemma}
\newtheorem{proposition}[theorem]{Proposition}

\theoremstyle{definition}
\newtheorem{definition}[theorem]{Definition}

\theoremstyle{remark}
\newtheorem{remark}[theorem]{Remark}
\newtheorem{question}[theorem]{Question}

\numberwithin{equation}{section}

\begin{document}

\hspace{1in} \vspace{-.3in}

\title[Cotorsion-free groups from a topological viewpoint]{Cotorsion-free  groups \\ from a topological viewpoint}

\author{Katsuya Eda}

\address{School of Science and Engineering, Waseda University, Tokyo 169-8555, JAPAN}

\email{eda@waseda.jp}

\author{Hanspeter Fischer}

\address{Department of Mathematical Sciences, Ball State University, Muncie, IN 47306, USA}

\email{fischer@math.bsu.edu}

\thanks{}



\date{August 30, 2014}

\commby{}

\begin{abstract} We present a characterization of cotorsion-free abelian groups in terms of homomorphisms from fundamental groups of Peano continua, which aligns naturally with the generalization of slenderness to non-abelian groups. In the process, we calculate the first homology group of the Griffiths twin cone.
\end{abstract}

 \maketitle

\section{Introduction and statement of main results}

The purpose of this paper is to establish new connections between two recent developments in ``wild'' algebraic topology
and to provide a new topological perspective on cotorsion-free abelian groups. Specifically, we give a characterization of cotorsion-free abelian groups in terms of homomorphisms from fundamental groups of Peano continua. In the process, we calculate the first homology group of the Griffiths twin cone.

Our results are stated in terms of normal subgroups $\pi({\mathcal U},x)$ of the fundamental group $\pi_1(X,x)$ with respect to open covers $\mathcal U$ of $X$ that first appeared in \cite[\S2.5]{Spanier} and have since come into renewed focus: $\pi({\mathcal U},x)$ is generated by all elements of the form $[\alpha\cdot \beta \cdot \alpha^-]$ with a path $\alpha:([0,1],0)\rightarrow (X,x)$, an open set $U\in {\mathcal U}$ such that $\alpha(1)\in U$,  and a loop $\beta:([0,1],\{0,1\})\rightarrow (U,\alpha(1))$, where ~$\cdot$ denotes path concatenation and $\alpha^-(t)=\alpha(1-t)$ denotes the reverse of the path $\alpha$.

These subgroups have been playing prominent roles in two different contexts: the generalization of slender groups and the generalization of covering spaces.

\subsection{Noncommutatively slender groups} A torsion-free abelian group $A$ is said to be {\em slender} if for every homomorphism $h: \mathbb{Z}^\mathbb{N}\rightarrow A$ there is an $n\in \mathbb{N}$ such that $h((c_k)_{k\in \mathbb{N}})=0$ whenever $c_k=0$ for all $k<n$.
The slender groups form a subclass of the cotorsion-free groups: an abelian group is slender if and only if it is cotorsion-free and contains no subgroup isomorphic to $\mathbb{Z}^\mathbb{N}$ \cite{Nunke}.
Recall that an abelian group $A$ is called {\em cotorsion} provided that whenever $A$ is a subgroup of an abelian group $B$ with $B/A$ torsion-free, we have $B=A\oplus C$ for some subgroup $C$ of $B$. In turn, $A$ is called {\em cotorsion-free} if it does not contain a nonzero cotorsion subgroup.

The concept of slenderness can be generalized to non-abelian groups by replacing  $\mathbb{Z}^\mathbb{N}$ with the fundamental group $\pi_1(\mathbb{H},{\bf o})$ of the Hawaiian Earring $\mathbb{H}$, that is, the  subspace of the Euclidean plane comprised of the union of the circles $C_k=\{(x,y)\in \mathbb{R}^2\mid x^2+(y-1/k)^2=1/k^2\}$ ($k\in \mathbb{N}$) accumulating at the origin ${\bf o}=(0,0)$.
Accordingly, a group $G$ is called  {\em noncommutatively slender} (or {\em n-slender} for short) if for every homomorphism  $h:\pi_1(\mathbb{H},{\bf o})\rightarrow G$ there is an $n\in \mathbb{N}$ such that $h([\gamma])=1$ for all loops $\gamma$ in $\bigcup_{k=n}^\infty C_k$. For example, every free group is n-slender \cite{Griffiths,Higman,MorganMorrison}. An abelian group is n-slender if and only if it is slender~\cite{Eda1992}. In general, we have the following characterization \cite{Eda2005}:  A group $G$ is n-slender if and only if for every Peano continuum $X$ and every homomorphism $h:\pi_1(X,x)\rightarrow G$, there is an open cover ${\mathcal U}$ of $X$ such that $h(\pi({\mathcal U},x))=1$.

The fundamental group $\pi_1(Y,y)$  of a path-connected topological space $Y$ is n\nobreakdash-slender if and only if every homomorphism $h:\pi_1(\mathbb{H},{\bf o})\rightarrow \pi_1(Y,y)$ is induced by a continuous map $f:(\mathbb{H},{\bf o})\rightarrow (Y,y)$ with $h=f_{\#}$ \cite{Eda1992}. The interplay of n-slenderness with free $\sigma$-products, as further investigated in \cite{Eda1998}, forms the foundation for the classification of the homotopy types of one-dimensional spaces by the isomorphism types of their fundamental groups \cite{Eda2010}.

\subsection{Generalized covering spaces}

Let $X$ be a path-connected topological space, $x\in X$, and let $Cov(X)$ denote the collection of all open covers of $X$.  Observe that ${\mathcal S}=\{\pi({\mathcal U},x)\mid {\mathcal U}\in Cov(X)\}$ is a collection of normal subgroups of $\pi_1(X,x)$ which is inversely directed by refinement: if $\mathcal V$ refines $\mathcal U$, then $\pi({\mathcal V},x)\leqslant \pi({\mathcal U},x)$.

We define \[\pi^s(X,x)
=\bigcap_{{\mathcal U}\in Cov(X)} \pi({\mathcal U},x)\] and call this normal subgroup of $\pi_1(X,x)$  the {\em Spanier group} of $X$.
 Provided $X$ is also locally path-connected, $X$ admits a universal covering space if and only if ${\mathcal S}$ contains a minimal element, i.e., $\pi^s(X,x)=\pi({\mathcal U},x)$ for some $\mathcal U$, and it admits a simply-connected covering space if and only if $\pi({\mathcal U},x)=1$ for some $\mathcal U$ \cite{Spanier}.

In general, $\pi^s(X,x)$ lies in the kernel of the natural homomorphism $\pi_1(X,x)\rightarrow \check{\pi}_1(X,x)$ to the first \v{C}ech homotopy group \cite{FZ2007} and it has recently been shown to equal this kernel if $X$ is locally path-connected and paracompact Hausdorff (e.g. if $X$ is a Peano continuum) \cite{BrazasFabel}. This homomorphism is often injective. For example, it is injective for one-dimensional spaces \cite{EdaKawamura1998}, for subsets of surfaces \cite{FZ2005} and for certain trees of manifolds \cite{FG}. Hence, for such $X$, we have $\pi^s(X,x)=1$.

The  Hawaiian Earring is the prototypical example of a Peano continuum that does not admit a universal covering space: $\pi^s(\mathbb{H},{\bf o})=1$, while $\pi({\mathcal U},{\bf o})\not=1$ for all ${\mathcal U}\in Cov(\mathbb{H})$.

It was shown in \cite{FZ2007}, that every path-connected topological space $X$ admits a {\em generalized covering projection} $p:\widetilde{X}\rightarrow X$ corresponding to $\pi^s(X,x)$, i.e.:
 \begin{itemize} \item[(i)] $\widetilde{X}$ is path-connected and locally path-connected; \item[(ii)] $p_\#:\pi_1(\widetilde{X},\widetilde{x})\rightarrow \pi_1(X,x)$ is a monomorphism onto $\pi^s(X,x)$; and \item[(iii)] for every map $f:(Y,y)\rightarrow (X,x)$ from a path-connected and locally path-connected space $Y$ with $f_\#(\pi_1(Y,y))\leqslant p_\#(\pi_1(\widetilde{X},\widetilde{x}))$, there is a {\sl unique} lift $\widetilde{f}:(Y,y)\rightarrow (\widetilde{X},\widetilde{x})$ such that $f=p\circ \widetilde{f}$.\end{itemize}
These properties uniquely characterize the concept, although they do not guarantee evenly covered neighborhoods or homeomorphic fibers. However, the automorphism group of this generalized covering projection is always naturally isomorphic to the quotient $\pi_1(X,x)/\pi^s(X,x)$ and it acts freely and transitively on every fiber. Moreover, $p$ is open if $X$ is locally path-connected \cite{FZ2007}. When $\pi^s(X,x)=1$, we speak of a {\em generalized universal covering}.

In the case of a one-dimensional compact metric space, the resulting generalized universal covering space carries a combinatorial $\mathbb{R}$-tree structure that acts as a {\em generalized} Cayley graph for the fundamental group \cite{FZ2012}. This has given rise to a ``mechanical'' description of the fundamental group of the Menger universal curve in terms of an infinite version of the Towers of Hanoi puzzle \cite{FZ2013}. Generalized universal coverings are also useful in determining the asphericity of other locally non-trivial spaces \cite{F}.

\pagebreak

\subsection{Main results}

\begin{definition} We call a group $G$ {\em homomorphically Hausdorff} relative to a path-connected topological space $X$ if for every homomorphism $h:\pi_1(X,x)\rightarrow G$, we have $\bigcap_{{\mathcal U}\in Cov(X)} h(\pi({\mathcal U},x))=1$.\end{definition}

\begin{remark}  The terminology is motivated by considering $\pi_1(X,x)$ as a topological group with basis $\{g\pi({\mathcal U},x)\mid g\in \pi_1(X,x)$, ${\mathcal U}\in Cov(X)\}$, as is done in \cite[\S3.3]{BDLM}, where this topology is called the {\em lasso topology}. (See also \cite{VZ}.) Given a homomorphism $h:\pi_1(X,x)\rightarrow G$, the image $K=h(\pi_1(X,x))$ is a topological group with basis $\{k h(\pi({\mathcal U},x))\mid k\in K$, ${\mathcal U}\in Cov(X)\}$ and $h:\pi_1(X,x)\rightarrow K$ is a quotient map. Then $\bigcap_{{\mathcal U}\in Cov(X)} h(\pi({\mathcal U},x))=1$ if and only if $K$ is Hausdorff.
 \end{remark}

\begin{definition} We call a group $G$ {\em Spanier-trivial} relative to a path-connected space $X$ if for every homomorphism $h:\pi_1(X,x)\rightarrow G$, we have $h(\pi^s(X,x))=1$.\end{definition}

\begin{remark}  There are some obvious relationships. Every group is Spanier-trivial relative to the Hawaiian Earring $\mathbb{H}$. If $G$ is homomorphically Hausdorff relative to $X$, then $G$ is Spanier-trivial relative to $X$. The fundamental group $\pi_1(X,x)$ is Spanier-trivial relative to $X$ if and only if $\pi^s(X,x)=1$, if and only if $\pi_1(X,x)$ is Hausdorff, in which case $X$ admits a generalized universal covering space and is {\em homotopically Hausdorff}, i.e., no fixed element  $g\in \pi_1(X,x)\setminus\{1\}$ can be represented by arbitrarily small loops.
\end{remark}

The prototypical example of a Peano continuum that is not homotopically Hausdorff is the Griffiths twin cone $C(\mathbb{H}_o)\vee C(\mathbb{H}_e)$: it is defined as the wedge of the two cones $C(\mathbb{H}_o)$ and $C(\mathbb{H}_e)$, over the subsets $\mathbb{H}_o=\bigcup_{k\in \mathbb{N}} C_{2k-1}$ and $\mathbb{H}_e=\bigcup_{k\in \mathbb{N}} C_{2k}$  of the Hawaiian Earring $\mathbb{H}$, respectively, joined at the distinguished points of their bases \cite{Griffiths1954}.
Since every loop in $C(\mathbb{H}_o)\vee C(\mathbb{H}_e)$ can be homotoped arbitrarily closely to the wedge point $\ast$ (see, e.g., \cite[Lemma~2.1]{Eda1991}), we have the following well-known fact (cf. \cite[\S2.5 Example 18]{Spanier}): $\pi^s(C(\mathbb{H}_o)\vee C(\mathbb{H}_e),\ast)=\pi_1(C(\mathbb{H}_o)\vee C(\mathbb{H}_e),\ast)\not=1$.

In particular, if a group $G$ is Spanier-trivial relative to $C(\mathbb{H}_o)\vee C(\mathbb{H}_e)$, then there is no nontrivial homomorphism $h:\pi_1(C(\mathbb{H}_o)\vee C(\mathbb{H}_e),\ast)\rightarrow G$, so that $G$ is also homomorphically Hausdorff relative to $C(\mathbb{H}_o)\vee C(\mathbb{H}_e)$.\vspace{5pt}

Here are our main results:

\begin{theorem}\label{mainresult}
For an abelain group $A$, the following are equivalent:
\begin{itemize}
\item[(1)] $A$ is cotorsion-free.
\item[(2)] $A$ is homomorphically Hausdorff relative to every Peano continuum.
\item[(3)] $A$ is homomorphically Hausdorff relative to the Hawaiian Earring $\mathbb{H}$.
\item[(4)] $A$ is Spanier-trivial relative to the Griffiths twin cone $C(\mathbb{H}_o)\vee C(\mathbb{H}_e)$.
\end{itemize}
\end{theorem}

The proof of Theorem~\ref{mainresult} will be presented in three separate sections: ``(1)$\Rightarrow$(2)'' in \S\ref{ncot} (Theorem~\ref{general}), ``(4)$\Rightarrow$ (1)'' in \S\ref{Griffiths} (Corollary~\ref{ST}), and ``(3)$\Rightarrow$(1)'' in \S\ref{hom} (Corollary~\ref{HE}). The main work of \S\ref{Griffiths} goes into proving the following:

\begin{theorem}\label{GH}
$H_1(C(\mathbb{H}_o)\vee C(\mathbb{H}_e))$  is isomorphic to \[\left(\bigoplus_{2^{\aleph_0}}\mathbb{Q}\right) \bigoplus \left(\prod_{p\in \mathbb{P}} A_p\right),\] where $\mathbb{P}$ is the set of all primes and $A_p$ is the $p$-adic completion of $\bigoplus_{2^{\aleph_0}}\mathbb{J}_p$.
\end{theorem}

\begin{remark} Note that Theorem~\ref{GH} is stated in the format of Kaplansky and that $A_p   \cong (\mathbb{J}_p)^\mathbb{N}$ \cite[pp.167--169]{Fuchs}. Moreover, by a theorem due to Balcerzyk (see \cite[VII.42 Exercise 7]{Fuchs}), we have $\big(\bigoplus_{2^{\aleph_0}}\mathbb{Q}\big) \bigoplus \big(\prod_{p\in \mathbb{P}} A_p\big)\cong\mathbb{Z}^\mathbb{N}/\bigoplus_{\mathbb{N}}\mathbb{Z}$.
\end{remark}

Theorem~\ref{mainresult} leaves open some natural questions regarding non-abelian groups:

\begin{question}
(a) Is there a group-theoretic characterization for the class of all groups that are homomorphically Hausdorff relative to every Peano continuum?
(b)  For which classes of groups do we have ``(3)$\Rightarrow$(2)'' or ``(4)$\Rightarrow$(2)''?
\end{question}

We list some non-abelian examples in Section~\ref{hom}.

\section{Some algebraic preliminaries}

In this section, we briefly recall some algebraic preliminaries from \cite{Fuchs} concerning infinite abelian groups.

 The {\em $\mathbb{Z}$-adic completion} of an abelian group $A$ is defined to be the inverse limit
$\displaystyle \widehat{A}=\lim_{\longleftarrow} (A/nA,\pi^m_n,n\in \mathbb{N})$
whose homomorphisms $\pi^m_n:A/mA\rightarrow A/nA$ are given by $\pi^m_n(a+mA)=a+nA$ for $n,m\in \mathbb{N}$ with $n|m$. Here, $nA=\{na\mid a\in A\}$. The kernel of the canonical map $A\rightarrow \widehat{A}$, given by $a\mapsto (a+nA)_{n\in \mathbb{N}}$, is called the {\em first Ulm subgroup} of $A$, and is denoted by $U(A)$. If we restrict $n$ to powers of a fixed prime $p$ in this inverse limit, we obtain the {\em $p$-adic completion} of $A$. For example, the $p$-adic integers $\mathbb{J}_p$ are defined to be the $p$-adic completion of $\mathbb{Z}$:
$\displaystyle \mathbb{J}_p=\lim_{\longleftarrow} \mathbb{Z}/p^k\mathbb{Z}$.

An abelian group $D$ is called {\em divisible} if for all $d\in D$ and all $n\in \mathbb{N}$ we have $n|d$, i.e., $d=nc$ for some $c\in D$.
Every divisible group is a direct sum of groups each isomorphic to $\mathbb{Q}$ or the quasicyclic group of type $p^\infty$ for some prime $p$
  (i.e., the subgroup of the complex multiplicative group $\mathbb{C}^\times$ consisting of all $p^n$-th roots of unity for all $n\geqslant 0$) \cite[(23.1)]{Fuchs}.
An abelian group $D$ is divisible if and only if it has the following property: whenever $D$ is a subgroup of an abelian group $A$, then $A=D\oplus C$ for some subgroup $C$ of $A$.
 Every abelian group $A$ has a maximal divisible subgroup $D$ (which is contained in $U(A)$) and we call $A$ {\em reduced}, if it does not contain a nonzero divisible subgroup \cite[(21.3)]{Fuchs}. It is an elementary exercise to show that the maximal divisible subgroup of a torsion-free abelian group is equal to its first Ulm subgroup $U(A)$, since in this case $U(A)$ itself is divisible.

  A subgroup $A$ of an abelian group $B$ is called {\em pure} if for every $a\in A$, we have $n|a$ in $A$ whenever $n|a$ in $B$. An abelian group $C$ is called {\em algebraically compact} if it has the following property: whenever $C$ is a pure subgroup of an abelian group $A$, then $A=C\oplus D$ for some subgroup $D$ of $A$.
  Clearly, every divisible abelian group is algebraically compact.
 Also, every finite abelian group is algebraically compact \cite[(38.1) and (3.1)]{Fuchs}. Every inverse limit of reduced algebraically compact abelian groups is reduced and algebraically compact \cite[(39.4)]{Fuchs}. In particular, $\widehat{\mathbb{Z}}$ and $\mathbb{J}_p$ are algebraically compact for any prime~$p$.

 An abelian group $C$ is divisible (respectively algebraically compact) if and only if it is {\em injective} (respectively {\em pure-injective}), i.e., every homomorphism $\tau:A\rightarrow C$ defined on a subgroup (respectively pure subgroup) $A$ of an abelian group $B$ extends to a homomorphism $\phi:B\rightarrow C$ \cite[(24.5) and (38.1)]{Fuchs}.

An abelian group is cotorsion if and only if it is the homomorphic image of an algebraically compact abelian group \cite[(54.1)]{Fuchs}. A torsion-free abelian group is cotorsion if and only if it is algebraically compact \cite[(54.5)]{Fuchs}. Every nonzero torsion-free reduced algebraically compact abelian group contains a direct summand isomorphic to $\mathbb{J}_p$
 \cite[(40.4)]{Fuchs}.
We therefore have the following characterization \cite{GoebelWald}:

\begin{theorem}[G\"obel-Wald] \label{goebel} An abelian group $A$ is cotorsion-free if and only if it is torsion-free and contains no subgroups isomorphic to $\mathbb{Q}$ or $\mathbb{J}_p$  for any prime $p$.
\end{theorem}

 An abelian group $A$ is said to be {\em complete modulo the first Ulm subgroup}, or {\em complete mod-U}, if for every sequence $(a_n)_{n\in \mathbb{N}}$ in $A$ with $(n+1)!|(a_{n+1}-a_n)$ for all $n\in \mathbb{N}$, there is an element $a \in A$ such that $(n+1)!|(a-a_n)$ for all $n\in \mathbb{N}$. An abelian group $A$ is algebraically compact if and only if $A$ is complete mod-U and the maximal divisible subgroup of $A$ equals $U(A)$ \cite[Satz 2.2]{DugasGoebel}. Hence, a torsion-free abelian group $A$ is algebraically compact if and only if it is complete mod-U.

\section{Cotorsion-free groups and Peano continua}\label{ncot}

\begin{theorem}\label{general}
Let $A$ be a cotorsion-free abelian group and $X$ a Peano continuum. Then $A$ is homomorphically Hausdorff relative to $X$.
\end{theorem}

We use tools from \cite{Eda2014}. For completeness, we give a self-contained proof.

\begin{proof}
Suppose, to the contrary, that there is a homomorphism $h:\pi_1(X,x)\rightarrow A$ and an element $0\neq a\in \bigcap_{{\mathcal U}\in Cov(X)} h(\pi({\mathcal U},x))$. Since $A$ is torsion-free, its first Ulm subgroup $U(A)$ equals the maximal divisible subgroup of $A$. However, $A$ is reduced, so that $U(A)=0$. Consider the $\mathbb{Z}$-adic completion $\widehat{\mathbb{Z}}$ of the integers $\mathbb{Z}$:
\[\widehat{\mathbb{Z}}=\lim_{\longleftarrow}\left(\mathbb{Z}/2!\,\mathbb{Z}\leftarrow \mathbb{Z}/3!\,\mathbb{Z} \leftarrow \mathbb{Z}/4!\,\mathbb{Z}\leftarrow \cdots\right).\]
An element $\widehat{u}\in  \widehat{\mathbb{Z}}$ can be  represented in the form \[\widehat{u}=(u_1+2!\,\mathbb{Z}, u_1+2!\,u_2+3!\,\mathbb{Z}, u_1+2!\,u_2+3!\,u_3+4!\,\mathbb{Z}, \dots),\] which we abbreviate by a formal sum $\sum_{i=1}^\infty i!\,u_i$. This representation is unique if we require that $u_i\in \{0,1,2, \dots, i\}$. While we do not add such infinite sums, we have \[\sum_{i=1}^\infty i!\,u_i=\sum_{i=1}^\infty i!\,v_i\] in $\widehat{\mathbb{Z}}$ if and only if \[(n+1)!\;\mid\; \sum_{i=1}^n i!\,u_i-\sum_{i=1}^n i!\,v_i\] in $\mathbb{Z}$ for all $n\in \mathbb{N}$.

Below, we will show that for each sequence $(u_i)_{i\in\mathbb{N}}$, there is an $[\ell]\in \pi_1(X,x)$ such that \begin{equation}\label{goal} (n+1)!\;\mid\; h([\ell])-\sum_{i=1}^n i!\,u_ia\end{equation}  in $A$ for all $n\in \mathbb{N}$. Since $\bigcap_{n\in \mathbb{N}} (n+1)!A=U(A)=0$, we obtain a well-defined homomorphism $\phi:\widehat{\mathbb{Z}}\rightarrow A$ by the formula \[ \phi\left(\sum_{i=1}^\infty i!\,u_i\right)=h([\ell]).\] Note that $\phi({\bf 1})=a\neq 0$, where ${\bf 1}=\sum_{i=1}^\infty i!\,u_i$ with $u_1=1$ and $u_i=0$ for all $i\geqslant 2$.
Since  $\widehat{\mathbb{Z}}$ is algebraically compact, $\phi(\widehat{\mathbb{Z}})$ is a nonzero cotorsion subgroup of $A$; a contradiction. We now show how to find the loop $\ell$.

Since $X$ is a Peano continuum, there is a continuous surjection $f:[0,1]\rightarrow X$. Let $\mathcal V_1$ be a cover of $X$ by open path-connected subsets of $X$ such that diam$(U)<1$ for every $U\in \mathcal{V}_1$.
Choose $k_1\in \mathbb{N}$ such that for every integer $s_1$ with $0\leqslant s_1<k_1$, there is a $U_{(s_1)}\in {\mathcal V}_1$ with $f([\frac{s_1}{k_1},\frac{s_1+1}{k_1}])\subseteq U_{(s_1)}$. Then ${\mathcal U}_1= \{U_{(s_1)}\mid  0\leqslant s_1<k_1\}$ covers $X$.
Next, consider a collection ${\mathcal V}_2$ of open path-connected subsets of $X$ such that diam$(U)<1/2$ for all $U\in {\mathcal V}_2$ and such that for each integer $s_1$ with $0\leqslant s_1<k_1$, there is a subcollection ${\mathcal V}_2'$ of ${\mathcal V}_2$ with $f([\frac{s_1}{k_1},\frac{s_1+1}{k_1}])\subseteq \bigcup {\mathcal V}_2'\subseteq U_{(s_1)}$. Choose $k_2\in \mathbb{N}$ such that for all integers $s_1$ and $s_2$ with $0\leqslant s_i<k_i$, there is a $U_{(s_1,s_2)}\in {\mathcal V}_2$ with
$f([\frac{s_1}{k_1}+\frac{s_2}{k_1k_2},\frac{s_1}{k_1}+\frac{s_2+1}{k_1k_2}])\subseteq U_{(s_1,s_2)}\subseteq U_{(s_1)}$. Then  ${\mathcal U}_2=\{U_{(s_1,s_2)}\mid 0\leqslant s_i<k_i\}$ covers $X$. Inductively, we find a sequence  of positive integers $k_n$ and open covers ${\mathcal U}_n=\{U_s\mid s\in S_n\}$ of $X$, where  $S_n=\{(s_1,s_2,\cdots, s_n)\mid s_i\in \{0,1,2,\dots,k_i-1\}\}$, with the following properties:
\begin{itemize}
\item[(i)] For every $U\in {\mathcal U}_n$, we have that $U$ is path-connected and diam$(U)<1/n$;
\item[(ii)] For every $s\in S_n$, we have $f([a_s,a_{s^+}])\subseteq U_s$, where $a_s=\sum_{i=1}^ns_i/\prod_{j=1}^i k_j$ and $s^+=(s_1, s_2, \cdots, s_n+1)$;
\item[(iii)] For every $(s_1, s_2, \dots, s_n)\in S_n$, we have $U_{(s_1,s_2,\dots,s_n)}\subseteq U_{(s_1, s_2, \dots, s_{n-1})}$.
\end{itemize}
For each $n\in \mathbb{N}$, we have $u_na\in h(\pi({\mathcal U}_n,x))$. Hence, for each $s\in S_n$, there are continuous paths $\alpha_{s,i}:([0,1],0)\rightarrow (X,x)$, say $1\leqslant i \leqslant r_s$, and continuous loops $\beta_{s,i}:([0,1],\{0,1\})\rightarrow (U_s,\alpha_{s,i}(1))$ (possibly constant) such that $u_na=h(\prod_{s\in S_n} \prod_{i=1}^{r_s} [\alpha_{s,i} \cdot \beta_{s,i}\cdot  \alpha_{s,i}^-])$. (Note that the order of the product does not matter since $A$ is abelian.)
 For each $s\in S_n$, choose paths $\gamma_{s,i}:[0,1]\rightarrow U_s$ with $\gamma_{s,i}(0)=f(a_s)$ and $\gamma_{s,i}(1)=\beta_{s,i}(0)$. Let $\ell_s$ be the loop $\gamma_{s,1}\cdot \beta_{s,1}\cdot \gamma_{s,1}^-\cdot \gamma_{s,2}\cdot \beta_{s,2}\cdot \gamma_{s,2}^-\cdot \;\cdots\; \cdot \gamma_{s,r_s}\cdot \beta_{s,r_s}\cdot\gamma_{s,r_s}^-$. Then $\ell_s$ is a loop in $U_s$ based at $f(a_s)$.  Let $\psi:\pi_1(X,x)\rightarrow H_1(X)$ denote the Hurewicz homomorphism. Since $A$ is abelian, we have a homomorphism $\widetilde{h}:H_1(X)\rightarrow A$ with $h=\widetilde{h}\circ \psi$. Let $\lfloor \ell_s\rceil\in H_1(X)$ denote the element represented by the cycle $\ell_s$. Then $u_na=h(\prod_{s\in S_n} \prod_{i=1}^{r_s} [\alpha_{s,i} \cdot \beta_{s,i} \cdot \alpha_{s,i}^-])=\widetilde{h}(\sum_{s\in S_n} \lfloor \ell_s\rceil )$.

We will now define a continuous path $g:[0,1]\rightarrow X$ from $g(0)=f(0)$ to $g(1)=f(1)$ so that the loop $\ell=g\cdot f^{-}$ has the following property: for each $n\in \mathbb{N}$,  $\ell$ runs precisely $i!$ many times through each loop $\ell_s$ with $s\in S_i$ and $1\leqslant i \leqslant n$, and the sum of the remaining subpaths of $\ell$ is homologous to $n!$ copies of the same cycle.
Specifically, put $T_n=\{(t_1, t_2, \dots, t_n)\mid 0\leqslant t_i < ik_i\}$ and $C_n=\{(c_1, c_2, \dots, c_n)\mid 0\leqslant c_i < i\}$. For each $t=(t_1, t_2, \dots, t_n)\in T_n$, let $s(t)=(s_1, s_2, \dots, s_n)\in S_n$ and $c(t)=(c_1, c_2, \dots,c_n)\in C_n$ be defined by the equation \[t_i=is_i+c_i\] and put \[b_t= \displaystyle \sum_{i=1}^n\frac{1+(3i-1)s_i+3c_i}{\prod_{j=1}^i (3j-1)k_j}.\]
Put $\epsilon_n=1/\prod_{i=1}^n (3i-1)k_i$. If we order the elements of $T=\bigcup_{n=1}^\infty T_n$ lexicographically, then the assignment $T\rightarrow [0,1]$ given by $t\mapsto b_t$ is strictly increasing.  Moreover, for each $n\in\mathbb{N}$ and each $t\in T_n$ with $c(t)=(c_1, c_2, \dots, c_n)$:
\begin{itemize}
\item[(i)] We have $(b_t-\epsilon_n,b_t)\cap \{b_{t'}\mid t'\in T\}=\emptyset$ and we define \[g|_{[b_t-\epsilon_n,b_t]}\equiv l_{s(t)}.\]
\item[(ii)] If $c_n<n-1$, we have $b_{t^+}=b_t+3\epsilon_n$ and $[b_t+\epsilon_n, b_t+2\epsilon_n]\cap \{b_{t'}\mid t'\in T\}=\emptyset$, and we define \[g|_{[b_t+\epsilon_n, b_t+2\epsilon_n]}\equiv (f|_{[a_{s(t)}, a_{s(t)^+}]})^-.\]
\end{itemize}
Since the loop $l_s$ is based at $f(a_s)$, $g$ is well-defined. On one hand, if $x\in [0,1]$ is such that $g(x)$ is not defined, then $x\in \{1\}\cup\bigcap_{n\in \mathbb{N}}\bigcup_{t\in T_n} (b_t,b_t+\epsilon_n)$. On the other hand, for every $n\in \mathbb{N}$, every $t\in T_n$ and every $x\in [b_t,b_t+\epsilon_n]$ such that $g(x)$ is defined, we have $g(x)\in U_{s(t)}$. Hence, $g$ uniquely extends to a continuous function $g:[0,1]\rightarrow X$ with $g(0)=f(0)$ and $g(1)=f(1)$.

Now fix $n\in \mathbb{N}$. We wish to decompose the homology cycle $\ell$ into appropriate 1-chains and rearrange them into smaller cycles.
For each $t\in T_m$ with $1\leqslant m \leqslant n$, clearly  $g|_{[b_t-\epsilon_m,b_t]}\equiv \ell_{s(t)}$ is a cycle itself. This leaves us with the 1-chains $g|_{[b_t, b_t+\epsilon_n]}$ for $t\in T_n$, the 1-chains $g|_{[b_t+\epsilon_m,b_t+2\epsilon_m]}=(f|_{[a_{s(t)}, a_{s(t)^+}]})^-$ for $t\in \bigcup_{m=2}^n T_m$ where $c(t)=(c_1, c_2, \dots, c_m)$ and $c_m<m-1$, and the second half of the loop $\ell=g\cdot f^-$. If $t\in T_n$ and $c(t)=(c_1, c_2, \dots, c_n)$ with $c_n<n-1$, we see that $g|_{[b_t, b_t+2\epsilon_n]}\equiv g|_{[b_t, b_t+\epsilon_n]} \cdot (f|_{[a_{s(t)},a_{s(t)^+}]})^-$ is trivially a cycle. In general, however, we need to regroup these 1-chains, which we do next.

 For each $t\in T_n$, define a sequence $t^\ast$ as follows: First, express $t=(t_1, t_2, \dots, t_n)$ and $c(t)=(c_1, c_2, \dots, c_n)$.   If $c_i=i-1$ for all $1\leqslant i \leqslant n$, we let $t^\ast=()$ be the empty sequence; otherwise, there is a unique $m\geqslant 2$ with $c_m<m-1$ and $c_i=i-1$ for all $m< i \leqslant n$, in which case we put $t^\ast=(t_1, t_2, \dots, t_m)$. (Note that $t^\ast=t$ corresponds to the case $c_n<n-1$, which gave us the trivial grouping above.)

Observe that for $t,\tilde{t}\in T_n$ with $t^\ast=\tilde{t}^\ast$ and $t\neq \tilde{t}$, we have $s(t)\neq s(\tilde{t})$. More precisely,  we have a bijection  \begin{equation} \label{bij} \{t\in T_n\mid t^\ast=()\;\}\rightarrow S_n: t\mapsto s(t), \end{equation} and
if we put $T'_m=\{t\in T_m\mid c_m<m-1$, where $c(t)=(c_1, c_2,\dots, c_m)\}$, then for every $u\in \bigcup_{m=2}^n T_m'$, we have a bijection
\begin{equation} \label{bijm} \{t\in T_n\mid t^\ast=u\}\rightarrow \{s\in S_n\mid a_{s(u)}\leqslant a_s< a_{s(u)^+}\}: t\mapsto s(t).\end{equation}

The correspondence (\ref{bij}) allows us to subdivide the second half of $\ell=g\cdot f^-$ and group each piece $(f|_{[a_s,a_{s^+}]})^-$ (for $s\in S_n$) with the 1-chain $g|_{[b_t, b_t+\epsilon_n]}$ where $t\in T_n$, $t^\ast=()$ and $s(t)=s$.
Similarly, the correspondence (\ref{bijm}) allows us to subdivide every $g|_{[b_u+\epsilon_m, b_u+2\epsilon_m]}\equiv (f|_{[a_{s(u)}, a_{s(u)^+}]})^-$
for which $u\in \bigcup_{m=2}^n T_m'$ and group each piece $(f|_{[a_s,a_{s^+}]})^-$ (for $s\in S_n$, $a_{s(u)}\leqslant a_s< a_{s(u)^+}$) with the 1-chain $g|_{[b_t, b_t+\epsilon_n]}$ where $t\in T_n$, $t^\ast=u$ and $s(t)=s$.

Now, for every $s\in S_n$, we have $|\{t\in T_n\mid s(t)=s\}|=n!$ and for all $t,\tilde{t}\in T_n$ with $s(t)=s(\tilde{t})$, we have $g|_{[b_t, b_t+\epsilon_n]}\equiv g|_{[b_{\tilde{t}}, b_{\tilde{t}}+\epsilon_n]}$. Hence,
  \[\lfloor g\cdot f^-\rceil=\sum_{i=1}^{n} i!\sum_{s\in S_i}\lfloor \ell_s\rceil +n!\sum_{s\in S_n}\lfloor \delta_s\rceil= \sum_{i=1}^{n-1} i!\sum_{s\in S_i}\lfloor \ell_s\rceil +n!\left(\sum_{s\in S_n}\lfloor \ell_s\rceil+\lfloor \delta_s\rceil\right)\]
  in $H_1(X)$, where $\delta_s\equiv g|_{[b_t, b_t+\epsilon_n]}\cdot (f|_{[a_s, a_s^+]})^-$ for $t\in T_n$ with $t^\ast=()$ and $s(t)=s$. Applying $\widetilde{h}$ yields (\ref{goal}).
\end{proof}

\section{The Griffiths twin cone}\label{Griffiths}

 In this section, we calculate the first integral homology group of the Griffiths twin cone. Our approach is based on \cite[\S4]{Eda1992}. However, there is a gap in the proof of \cite[Lemma~4.11]{Eda1992}, which is addressed in \cite{Eda2014}. Lemma~\ref{form} below generalizes the corresponding adjustment of \cite[Lemma~4.11]{Eda1992} from \cite{Eda2014} to the situation at hand. In keeping with a geometric perspective, the proofs in this section are framed in terms of the generalized universal covering  of the Hawaiian Earring, treating infinitary words implicitly.

Two applications of van Kampen's Theorem (each time cutting off one simply-connected cone tip) yield $\pi_1(C(\mathbb{H}_o)\vee C(\mathbb{H}_e))\cong \pi_1(\mathbb{H})/N_0$, where $N_0$ is the normal subgroup of $\pi_1(\mathbb{H})$ generated by $\pi_1(\mathbb{H}_o)\ast \pi_1(\mathbb{H}_e)$ \cite[\S3]{Griffiths1954}; we take the origin ${\bf o}$ as the base point for $\mathbb{H}$. Hence, $H_1(C(\mathbb{H}_o)\vee C(\mathbb{H}_e)) \cong \pi_1(\mathbb{H})/N_1$, where $N_1=(\pi_1(\mathbb{H}_o)\ast \pi_1(\mathbb{H}_e))\pi_1(\mathbb{H})'$ and $\pi_1(\mathbb{H})'$ denotes the commutator subgroup of $\pi_1(\mathbb{H})$.

Consider the generalized universal covering $p:\widetilde{\mathbb{H}}\rightarrow \mathbb{H}$. Then $\widetilde{\mathbb{H}}$ is an $\mathbb{R}$-tree \cite[Example~4.14]{FZ2007}, i.e., $\widetilde{\mathbb{H}}$ is a uniquely arcwise connected geodesic metric space. (Recall that an isometric embedding of a compact interval of the real line into a metric space is called a {\em geodesic}. A {\em geodesic metric space} is a metric space in which every pair of points is connected by a geodesic.) In particular, $\widetilde{\mathbb{H}}$ is simply-connected.
Moreover, for every $\widetilde{x}\in p^{-1}({\bf o})\subseteq \widetilde{\mathbb{H}}$  and every $b\in \pi_1(\mathbb{H})$
there is a unique $\widetilde{y}\in p^{-1}({\bf o})$
such that the  arc $[\widetilde{x},\widetilde{y}]$ in $\widetilde{\mathbb{H}}$
from $\widetilde{x}$ to $\widetilde{y}$, when projected by $p:\widetilde{\mathbb{H}}\rightarrow \mathbb{H}$, is a loop representing $b$.

\begin{lemma}\label{nooverlap}
  Given arcs $[\widetilde{x}_1,\widetilde{y}_1], [\widetilde{x}_2,\widetilde{y}_2]\subseteq \widetilde{\mathbb{H}}$  whose projections represent the same element of $\pi_1(\mathbb{H})$, there is a homeomorphism $h: \widetilde{\mathbb{H}}\rightarrow \widetilde{\mathbb{H}}$ such that $p\circ h=p$, $h(\widetilde{x}_1)=\widetilde{x}_2$ and $h(\widetilde{y}_1)=\widetilde{y}_2$. If, moreover, $[\widetilde{x}_2,\widetilde{y}_2]\subseteq [\widetilde{x}_1,\widetilde{y}_1]$, then $[\widetilde{x}_2,\widetilde{y}_2]=[\widetilde{x}_1,\widetilde{y}_1]$.
\end{lemma}

\begin{proof} The first part follows from the lifting property of $p:\widetilde{\mathbb{H}}\rightarrow \mathbb{H}$. Considering $h^2$ if necessary, we may assume that $\widetilde{x}_1\leqslant h(\widetilde{x}_1)=\widetilde{x}_2\leqslant\widetilde{y}_2=h(\widetilde{y}_1)\leqslant\widetilde{y}_1$ on $[\widetilde{x}_1,\widetilde{y}_1]$. Then the nondecreasing sequence $\widetilde{x}_1, h(\widetilde{x}_1), h^2(\widetilde{x}_1),h^3(\widetilde{x}_1),\dots$ converges to some $\widetilde{x}\in [\widetilde{x}_1,\widetilde{y}_1]$ and each arc $[h^{i-1}(\widetilde{x}_1),h^{i}(\widetilde{x}_1)]$ projects to the same loop $\alpha$ in $\mathbb{H}$. The continuity of $p|_{[\widetilde{x}_1,\widetilde{y}_1]}$ implies that $\alpha$ is constant.
Thus, $\widetilde{x}_1=\widetilde{x}_2$, so that $h=id$. Consequently, $\widetilde{y}_1=\widetilde{y}_2$. Finally, note that $h^2=id$ implies $h=id$, since the automorphism group of $p:\widetilde{\mathbb{H}}\rightarrow \mathbb{H}$ is isomorphic to $\pi_1(\mathbb{H})$, and hence isomorphic to a subgroup of the torsion-free group $\check{\pi}_1(\mathbb{H})$.
\end{proof}

Suppose we have two elements $a$ and $b$  in  $\pi_1(\mathbb{H})$, represented by the projections of  arcs $[\widetilde{x},\widetilde{y}]$ and $[\widetilde{y},\widetilde{z}]$ in $\widetilde{\mathbb{H}}$, respectively. Then the projection of the concatenation $[\widetilde{x},\widetilde{y}]\cup [\widetilde{y},\widetilde{z}]$ represents the product $ab$, but it is in general not an arc. Since $\widetilde{\mathbb{H}}$ is uniquely arcwise connected, there is a unique
      $\widetilde{t}\in [\widetilde{x},\widetilde{y}]\cap [\widetilde{y},\widetilde{z}]$ such that the arc $[\widetilde{x},\widetilde{z}]$, whose projection also represents $ab$, satisfies $[\widetilde{x},\widetilde{z}]=[\widetilde{x},\widetilde{t}]\cup [\widetilde{t},\widetilde{z}]$. Note that $\widetilde{t}\in p^{-1}({\bf o})$, since  $p|_{\widetilde{\mathbb{H}}\setminus p^{-1}({\bf o})}:\widetilde{\mathbb{H}}\setminus p^{-1}({\bf o})\rightarrow \mathbb{H}\setminus\{{\bf o}\}$ is a local homeomorphism. Thus, we have:

\begin{lemma}\label{tree}
Let $g_1,g_2,\dots, g_s\in \pi_1(\mathbb{H})$ and $f:[a,b]\rightarrow \widetilde{\mathbb{H}}$ a path such that $f|_{[t_{i-1},t_i]}$ is a geodesic with $g_i=[p\circ f|_{[t_{i-1},t_i]}]$ for some subdivision $\{t_0, t_1, \dots, t_s\}$ of $[a,b]$. Then $Im(f)$ is homeomorphic to a finite simplicial tree whose vertices lie in $p^{-1}({\bf o})$. Standard edge cancellation yields a geodesic $\widehat{f}:[\widehat{a},\widehat{b}]\rightarrow \widetilde{\mathbb{H}}$ with $g_1g_2\cdots g_s=[p\circ \widehat{f}]$.
\end{lemma}

 A somewhat more delicate algorithm is necessary if we are given a product $g_1g_2\cdots g_s\in (\pi_1(\mathbb{H}_o)\ast \pi_1(\mathbb{H}_e))\pi_1(\mathbb{H})'$, whose factors $g_i$ either lie in $\pi_1(\mathbb{H}_o)$ or in $\pi_1(\mathbb{H}_e)$, or else are paired by inverses, and wish to end up with a similar pairing structure for the geodesic $\widehat{f}$. This is carried out in the proof of the following lemma, which is a generalization of the corresponding result for $\pi_1(\mathbb{H})'$ in \cite{Eda2014}:

\begin{lemma}\label{form}
Let $g\in(\pi_1(\mathbb{H}_o)\ast \pi_1(\mathbb{H}_e))\pi_1(\mathbb{H})'$. Then every geodesic $f:[a,b]\rightarrow \widetilde{\mathbb{H}}$  with $g=[p\circ f]$ has the following property:

$(\ast)$ There is a subdivision $\{t_0, t_1,\dots,t_s\}$ of $[a,b]$ defining loops $f_i=p\circ f|_{[t_{i-1},t_i]}$ in $\mathbb{H}$ based at $\bf o$, i.e., $g=[f_1][f_2]\cdots [f_s]$, along with a partition $F_o, F_e, C, \overline{C}$ of $\{1,2,\dots,s\}$ and a bijection $\varphi:C\rightarrow \overline{C}$ such that
\begin{itemize}
\item[(i)] for every $i\in F_o$, $f_i$ lies in $\mathbb{H}_o$;
\item[(ii)] for every $i\in F_e$, $f_i$ lies in $\mathbb{H}_e$;
\item[(iii)] for every $i\in C$, $f_{\varphi(i)}\equiv f_i^-$
\end{itemize}

\end{lemma}

\begin{proof} We may assume $g\not=1$.
 Let $\widetilde{x}\in p^{-1}({\bf o})$. Since $g\in(\pi_1(\mathbb{H}_o)\ast \pi_1(\mathbb{H}_e))\pi_1(\mathbb{H})'$,  there exists a  path $f:[a,b]\rightarrow \widetilde{\mathbb{H}}$ with $f(a)=\widetilde{x}$ and $g=[p\circ f]$, satisfying $(\ast)$, such that each $f|_{[t_{i-1},t_i}]$ is a geodesic. It suffices to show that the (unique) geodesic in $\widetilde{\mathbb{H}}$ from $f(a)$ to $f(b)$ also satisfies $(\ast)$. To this end, we let $r$ denote the number of indices $i\in \{1,2,\dots, s-1\}$ for which $f|_{[t_{i-1},t_{i+1}]}$ is not a geodesic. If $r=0$, then $f$ is a geodesic and we are done. Otherwise, we recursively subject $f$ to the transformation described in the following paragraph, which replaces $f$ with $f|_{[a,a_1]\cup[b_1,b]}$ for some $f|_{[a_1,t_i]}\equiv f|_{[t_i,b_1]}^-$, while retaining  property ($\ast$) on a refined  subdivision (which includes $a_1$ and $b_1$), eventually reducing the pair $(r,s)$ in the lexicographical ordering.

Let $i$ be the largest index for which $f|_{[t_{i-1},t_{i+1}]}$ is not a geodesic. Let $a_1\in [t_{i-1},t_i)$ and $b_1\in (t_i,b]$
be the unique points such that the arc $[f(t_{i-1}),f(b)]$ from $f(t_{i-1})$ to $f(b)$ in $\widetilde{\mathbb{H}}$ equals $[f(t_{i-1}),f(a_1)]\cup [f(b_1),f(b)]$. In particular, $f(a_1)=f(b_1)\in p^{-1}({\bf o})$ and $f|_{[a_1,t_i]}\equiv f|_{[t_i,b_1]}^-$. Say, $b_1\in (t_{j-1},t_j]$. Define $f_{(i,1)}=p\circ f|_{[t_{i-1},a_1]}$, $f_{(i,2)}=p\circ f|_{[a_1,t_i]}$, $f_{(j,1)}=p\circ f|_{[t_{j-1},b_1]}$, and $f_{(j,2)}=p\circ f|_{[b_1,t_j]}$. Then $f_i=f_{(i,1)}f_{(i,2)}$, $f_j=f_{(j,1)}f_{(j,2)}$, and $f_{(i,2)}=(f_{i+1}f_{i+2}\cdots f_{j-1} f_{(j,1)})^-$. If $a_1=t_{i-1}$, then $f_{(i,1)}$ is degenerate and we treat it as empty. Likewise, if $b_1= t_j$, we treat $f_{(j,2)}$ as empty.
If $i\in C$, then there is a (unique) subdivision $\eta=\{t_{\varphi(i)-1}<t^\ast_{i+1}<t^\ast_{i+2}<\cdots< t^\ast_j<t_{\varphi(i)}\}$ such that $f(t^\ast_k)=f(t_k)$ for all $i<k<j$ and $f(t^\ast_j)=f(b_1)$. (By Lemma~\ref{nooverlap}, either $\varphi(i)<i$ or $\varphi(i)\geqslant j$.)  If $j\in \overline{C}$, then there is a (unique)  $t^{\ast\ast}_j \in [t_{\varphi^{-1}(j)-1}, t_{\varphi^{-1}(j)})$ such that $f(t^{\ast\ast}_j)=f(b_1)$. Analogous statements hold if $i\in \overline{C}$ or $j\in C$. Provided $\varphi(i)\not= j$, we may thus reduce the domain of $f$ to $[a,a_1]\cup [b_1,b]$ and adjust the concatenation $f_1 f_2\cdots f_s$ to the new subdivision by replacing $f_i$ with $f_{(i,1)}$, replacing $f_{j}$ with $f_{(j,2)}$, then eliminating $f_{i+1}, f_{i+2},\dots,f_{j-1}$, then replacing $f_{\varphi(i)}$ with $\widehat{f}_{i+1}\widehat{f}_{i+2}\cdots \widehat{f}_{j-1}\widehat{f}_{(j,1)}\widehat{f}_{(i,1)}^-$ (where $\widehat{f}_m=p\circ f|\equiv f_m$ for all $m$), and  finally replacing $f_{\varphi^{-1}(j)}$ (which could by now be one of the new $\widehat{f}_{i+1}, \widehat{f}_{i+2}, \dots, \widehat{f}_{j-1}$) with $\widehat{f}_{(j,2)}^-\widehat{f}_{(j,1)}^-$. Observe that if $i\in F_o\cup F_e$, then $f_i, f_{i+1}, \dots, f_{j-1}$ and $f_{(j,1)}$ all lie in $\mathbb{H}_o$ or $\mathbb{H}_e$, so that there is no need to introduce $\widehat{f}_{i+1}\widehat{f}_{i+2}\cdots \widehat{f}_{j-1}\widehat{f}_{(j,1)}\widehat{f}_{(i,1)}^-$. Likewise, if $j\in F_o\cup F_e$, then there is no need to introduce $\widehat{f}_{(j,2)}^-\widehat{f}_{(j,1)}^-$.
If $\varphi(i)=j$,  we instead consider the common subdivision $\eta\cup\{b_1\}$ of $[t_{j-1},t_j]$. While in this case $[t_{j-1},b_1]$ (the domain of $f_{(j,1)}$) might overlap with $[t^\ast_{j-1},t^\ast_j]$ (the domain of $\widehat{f}_{(j,1)}$), the former cannot properly contain the latter by Lemma~\ref{nooverlap}. Consequently, we can transfer the points $\eta\cap [t_{j-1},b_1]$ to $[b_1,t_j]$ along the correspondence of $f_{(j,1)}$ with $\widehat{f}_{(j,1)}$ (repeatedly if necessary) and find a $k\in \{i+1,i+2,\dots,j-1\}$  with which we may adjust $f_1 f_2\cdots f_s$ by replacing $f_i$ with $f_{(i,1)}$, then eliminating $f_{i+1}, f_{i+2},\dots,f_{j-1}$, then  replacing $f_{\varphi(i)}=f_{j}$ with $\widehat{f}_{(k,2)} \widehat{f}_{k+1}\widehat{f}_{k+2}\cdots \widehat{f}_{j-1}\widehat{f}_{i+1}\widehat{f}_{i+2}\cdots \widehat{f}_{k-1}\widehat{f}_{(k,1)}\widehat{f}_{(i,1)}^-$, and finally (if applicable) replacing $f_{\varphi(k)}$ or $f_{\varphi^{-1}(k)}$ (or its new copy) by $\widehat{f}_{(k,2)}^-\widehat{f}_{(k,1)}^-$.
Then $f|_{[a,a_1]\cup[b_1,b]}$ satisfies $(\ast)$ with respect to the new subdivision.

While this transformation potentially increases $s$ by as much as 2, it may decrease~$r$. Specifically, if $f_{(i,1)}$ is not empty, then $r$ decreases by 1. So, let us assume that $f_{(i,1)}$ is empty. If  $f_{(j,2)}$ is also empty, then $r$ either decreases (by 1 or 2) or remains constant, but $s$ definitely decreases. So, assume  that $f_{(j,2)}$ is not empty. If $i=1$, then $r$ decreases from 1 to 0. So, assume  that $i>1$. Now the outcome depends on whether  $f|_{[t_{i-2},t_{i-1}]\cup [b_1,b]}$ is a geodesic or not. If it is, then $r$ decreases by 1. If it is not, then $r$ remains constant and $s$ does not increase. It now suffices to show that this final scenario cannot occur indefinitely, leaving both $r$ and $s$ unchanged, as we iterate the transformation for $f|_{[a,a_1]\cup[b_1,b]}$ and its subdivision. Suppose, to the contrary, that it does.

  We then have sequences $a<\cdots <a_{n+1}<a_n<\cdots <b_n<b_{n+1}<\cdots <b$\linebreak with subdivisions $\xi_n$ of $[a,a_n]\cup [b_n,b]$ into $s$  intervals, each obtained from the previous one by the above transformation, such that $f|_{[a,a_n]\cup [b_n,b]}$ satisfies~$(\ast)$ with pairings $\varphi_n:C_n\rightarrow \overline{C}_n$. In particular, $[a_{n+1},a_n]$ is a subinterval of $\xi_n$ and $f|_{[a_{n+1},a_n]}\equiv f|_{[b_n,b_{n+1}]}^-$.
  Put $a_\infty=\inf\{a_n\mid n\in \mathbb{N}\}$ and $b_\infty=\sup\{b_n\mid n\in \mathbb{N}\}$.

  We will call a subinterval of $\xi_n$ an
{\em inside} interval if it is contained in ${\mathcal I}=[a_\infty,b_\infty]$, an {\em outside} interval if it is contained in ${\mathcal O}=[a,a_\infty]\cup [b_\infty,b]$, and an {\em overlapping} interval otherwise.

Since $\xi_n\cap \left([a,a_{n+1}]\cup [b_{n+1},b]\right)\subseteq \xi_{n+1}$ and since the number of subintervals of $\xi_n$ is equal to $s$ for all $n$, the number of points in $\xi_n\cap {\mathcal O}$ is  nondecreasing and bounded by~$s$. So, we may assume that this number is constant. In particular, an overlapping interval $[u,v]$ of $\xi_n$ cannot be paired by $\varphi_n$ with an outside interval $[u',v']$ of $\xi_n$. (Otherwise, they would be paired subintervals of $\xi_{n+1}, \xi_{n+2}, \cdots, \xi_m$ until $\xi_m\cap(u,v)\not=\emptyset$ for some minimal $m>n$. But then $\xi_m\cap(u',v')\not=\emptyset$, according to our transformation rule, implying that $|\xi_m\cap {\mathcal O}|>|\xi_{m-1}\cap {\mathcal O}|$; a contradiction.)

It also follows that if a lower (respectively upper) overlapping interval persists, for all $n$, then its left (respectively right) endpoint is constant. Therefore, a persistent overlapping lower interval $[u,v_n]$ cannot be paired with an inside interval for infinitely many $n$. (Otherwise, there are $c,d\in [u,a_\infty]$ and $c_m,d_m\in [a_\infty,a_m]\cup [b_m,b_\infty]$ with $f(c_m)=f(c)\not=f(d)=f(d_m)$ for all $m$. However, since $f(a_m)=f(b_m)$ converges to $f(a_\infty)=f(b_\infty)$, so do $f(c_m)$ and $f(d_m)$.) The same is true for a persistent overlapping upper interval $[u_n,v]$. So, we may assume that overlapping intervals are not paired with inside intervals. (Note that an overlapping interval might cease to exist, in which case it will not reappear.) Similarly, we may assume that outside intervals are not paired with inside intervals. In summary, we now assume that inside, outside, and overlapping intervals, if paired, are paired with an interval of the same kind.

We claim that for every point $c_n$ contained in an inside interval of $\xi_n$, there is a point $c_{n+1}$ contained in an inside interval of $\xi_{n+1}$ such that $f(c_n)=f(c_{n+1})$. In order to show this, we may assume that $c_n\in [a_{n+1},a_n]\cup [b_n,b_{n+1}]$. (Otherwise we can take $c_{n+1}=c_n$.) We will use the notation from  our transformation rule for $C=C_n$. Since $f_{(i,1)}$ is empty, but $f_{(j,2)}$ is not, and since the number of subintervals of $\xi_n$ and $\xi_{n+1}$ are equal, we have $j\in C\cup \overline{C}$. Observe that the domain of $f_i$, i.e., $[a_{n+1},a_n]$, is an inside interval of $\xi_n$, while the domain of $f_j$ might be an upper overlapping interval, in which case the two would not be paired by $\varphi$.
 First suppose that $i\in C$.  Then the desired point $c_{n+1}$ can be found in the domain of $f_{\varphi(i)}$. The same is true if $i\in \overline{C}$. Consequently, the only case that needs attention is when $i\in F_o\cup F_e$. Say, $i\in F_o$. If $j>i+1$, then both $f_i$ and $f_{i+1}$ lie in $\mathbb{H}_o$, so that combining their domains will reduce the pair $(r,s)$ to $(r,s-1)$ for $\xi_n$. So, we may assume that $j=i+1$.  In this case, the domain of $f_{(j,1)}$ equals $[b_n,b_{n+1}]$ and $f_i\equiv f_{(j,1)}^-$. Since $|\xi_{n+1}\cap {\mathcal O}|=|\xi_{n}\cap {\mathcal O}|$, the domains of $f_{(j,1)}$ and $\widehat{f}_{(j,1)}$ must both be inside intervals of $\xi_{n+1}$, so that we can find $c_{n+1}$ in $\widehat{f}_{(j,1)}$.

Hence, starting with any two points $c_n$ and $d_n$  in any inside interval of some $\xi_n$  such that $f(c_n)\not= f(d_n)$, we obtain sequences $(c_m)_{m\geqslant n}$ and $(d_m)_{m\geqslant n}$ such that $c_m$ and $d_m$ are contained in (possibly different) inside intervals of $\xi_m$ with $f(c_m)=f(c_n)$ and $f(d_m)=f(d_n)$. This leads to the same contradiction as before.
\end{proof}

\begin{proof}[Proof of Theorem~\ref{GH}]
As in \cite{EdaKawamura2000}, it suffices to show that $H_1(C(\mathbb{H}_o)\vee C(\mathbb{H}_e))$ is (i)~torsion-free; (ii)~algebraically compact; (iii) contains a subgroup isomorphic to $\bigoplus_{2^{\aleph_0}}\mathbb{Q}$; and (iv) contains a pure subgroup isomorphic to $\bigoplus_{2^{\aleph_0}}\mathbb{Z}$.

The fact that $A=H_1(C(\mathbb{H}_o)\vee C(\mathbb{H}_e))$ is torsion-free follows from the formula $H_1(\mathbb{H}) \cong H_1(\mathbb{H}_o)\oplus H_1(\mathbb{H}_e)\oplus A$ \cite[Theorem~1.2] {Eda1991} and the fact that $H_1(\mathbb{H})$ is torsion-free \cite[Corollary~2.2]{EdaKawamura1998}. (Note that the torsion-freeness of $H_1(\mathbb{H})$ also follows from the fact that $\pi_1(\mathbb{H})$ is isomorphic to a subgroup of $\check{\pi}_1(\mathbb{H})$, all of whose finitely generated subgroups are free.)
 It follows from \cite[Corollary~9.2]{BZ} (or by adapting the proof of \cite[Theorem~4.14]{Eda1992}) that the maximal divisible subgroup of $A$ is isomorphic to $\bigoplus_{2^{\aleph_0}}\mathbb{Q}$.
 By \cite[Theorem~1.1]{Eda1991},  $A$ is complete mod-U and hence algebraically compact. Therefore,  Lemma~\ref{sort} below completes the proof.
\end{proof}

\begin{lemma}\label{sort}
$H_1(C(\mathbb{H}_o)\vee C(\mathbb{H}_e))$ contains a pure subgroup isomorphic to $\bigoplus_{2^{\aleph_0}}\mathbb{Z}$.
\end{lemma}

\begin{proof}
Consider the element $a\in \pi_1(\mathbb{H})$  represented by  $\ell_1\ell_2\ell_3\cdots=\ell:[0,1]\rightarrow \mathbb{H}$, where $\ell_i=\ell|_{[(i-1)/i,i/(i+1)]}$ canonically winds once around the circle $C_i$ of $\mathbb{H}$ for each $i$.
We first show that $aN_1$ generates a non-trivial pure subgroup of $\pi_1(\mathbb{H})/N_1$.
To this end, suppose that $a^m=b^nc$ for some $b\in \pi_1(\mathbb{H})$, $c\in N_1$, $m\geqslant 1$ and $n\geqslant 0$. We wish to show that $n>0$ and $n|m$. Since $\pi_1(\mathbb{H})/N_1$ is torsion-free, it suffices to show this for $m=1$. (For the general case, express $d=gcd(m,n)$ as $d=\alpha m + \beta n$ with $\alpha, \beta\in \mathbb{Z}$ and write $m=m_0d$, $n=n_0d$. Then $a^{m_0}=b^{n_0}c_0$ for some $c_0\in N_1$. Hence, $a=(b^\alpha a^\beta)^{n_0}c_o^\alpha$, implying $n_0=1$ and $n|m$.)

Let $[\widetilde{x}_a,\widetilde{y}_a], [\widetilde{x}_{b},\widetilde{y}_{b}]$ and $[\widetilde{x}_c,\widetilde{y}_c]$ be arcs of $\widetilde{\mathbb{H}}$ whose projections represent $a, b$ and $c$, respectively.
Let $\kappa_i^+(c)$, respectively $\kappa_i^-(c)$, denote the number of subarcs $[\widetilde{x},\widetilde{y}]$ of $[\widetilde{x}_c,\widetilde{y}_c]$ which project to $\ell_i\ell_{i+1}\ell_{i+2}\cdots$, respectively $(\ell_i\ell_{i+1}\ell_{i+2}\cdots)^-$. (Observe that two such arcs cannot overlap. Hence, as in the proof of Lemma~\ref{nooverlap}, this number is finite.) Analogously, we define $\kappa_i^\pm(b)$ and $\kappa_i^\pm(a)$. Applying  Lemma~\ref{tree} to $c=b^{-n}a$, we obtain $\kappa_i^+(c)-\kappa_i^-(c)=1-n\left(\kappa_i^+(b)-\kappa_i^-(b)\right)$, for sufficiently large $i$. However, by Lemma~\ref{form}, $\kappa_i^+(c)-\kappa_i^-(c)=0$, for sufficiently large $i$. Hence $n=1$.

We now vary this construction. Choose a collection $\{I_{\alpha}\mid \alpha\in\{1,2\}^\mathbb{N}\}$ of infinite subsets of $\mathbb{N}$ such that $I_\alpha\cap I_\beta$ is finite for all $\alpha\neq \beta$; e.g., for $\alpha=(s_k)_{k\in\mathbb{N}}$, take $I_{\alpha}=\{\sum_{k=1}^n s_k2^{n-k}\mid n\in \mathbb{N}\}$. Let $a_\alpha\in \pi_1(\mathbb{H})$ be the element represented by the loop $\ell_{2i_1-1}\ell_{2i_1}\ell_{2i_2-1}\ell_{2i_2}\ell_{2i_3-1}\ell_{2i_3}\cdots$, where $i_1<i_2<i_3<\cdots$ is an enumeration of $I_\alpha$. Then the set $\{a_\alpha N_1\mid \alpha\in \{1,2\}^\mathbb{N}\}$ is linearly independent in $\pi_1(\mathbb{H})/N_1$. Indeed, suppose $a_{\alpha_1}^{m_1}a_{\alpha_2}^{m_2}\cdots a_{\alpha_k}^{m_k}N_1=N_1$ for some $\alpha_i\in \{1,2\}^\mathbb{N}$ and $m_i\in \mathbb{Z}$. Choose $M\in \mathbb{N}$  so that $(I_{\alpha_i}\setminus\{1,2,\cdots,M\})\cap (I_{\alpha_j}\setminus\{1,2,\cdots,M\})=\emptyset$ for all $i\neq j$. Put $J_i=I_{\alpha_i}\setminus\{1,2,\cdots,M\}$ and let $q_i:\mathbb{H}\rightarrow \bigcup_{t\in J_i} C_t\subseteq \mathbb{H}$ denote  retraction  with $q_i(x)={\bf o}$ for $x\not\in \bigcup_{t\in J_i} C_t$. Then $q_i$ induces a homomorphism $q_{i\#}:\pi_1(\mathbb{H})/N_1\rightarrow \pi_1(\mathbb{H})/N_1$, which we may apply to the equation $a_{\alpha_1}^{m_1}a_{\alpha_2}^{m_2}\cdots a_{\alpha_k}^{m_k}N_1=N_1$ to obtain $q_{i\#}(a_{\alpha_i})^{m_i}N_1=N_1$. We conclude that $m_i=0$ for all $i\in\{1,2,\cdots,k\}$.

In order to show that $\left<a_\alpha N_1\mid \alpha\in \{1,2\}^\mathbb{N}\right>\cong \bigoplus_{2^{\aleph_0}}\mathbb{Z}$ is a pure subgroup of $\pi_1(\mathbb{H})/N_1$, suppose that  $a_{\alpha_1}^{m_1}a_{\alpha_2}^{m_2}\cdots a_{\alpha_k}^{m_k}N_1=b^nN_1$ for some $\alpha_i\in\{1,2\}^\mathbb{N}$, $m_i\in \mathbb{Z}\setminus\{0\}$, $b\in \pi_1(\mathbb{H})$ and $n\in \mathbb{N}$. Then $q_{i\#}(a_{\alpha_i})^{m_i}N_1=q_{i\#}(b)^nN_1$ for all $i$. As in the proof of the purity of $aN_1$ above, we conclude that $n|m_i$ for all $i$. Hence, $n|a_{\alpha_1}^{m_1}a_{\alpha_2}^{m_2}\cdots a_{\alpha_k}^{m_k}N_1$ in $\pi_1(\mathbb{H})/N_1$.
\end{proof}

\begin{corollary}\label{ST}
 Let $A$ be an abelian group which is Spanier-trivial relative to the Griffiths twin cone $C(\mathbb{H}_o)\vee C(\mathbb{H}_e)$. Then $A$ is cotorsion-free.
  \end{corollary}

  \begin{proof}
By Theorem~\ref{GH}, $\pi^s(C(\mathbb{H}_o)\vee C(\mathbb{H}_e),\ast)=\pi_1(C(\mathbb{H}_o)\vee C(\mathbb{H}_e),\ast)$ can be mapped homomorphically onto $\mathbb{Q}$, $\mathbb{J}_p$ and $\mathbb{Z}/p\mathbb{Z}$ for any prime $p$. It follows that $A$ cannot contain a subgroup isomorphic to any of these groups and is consequently cotorsion-free by Theorem~\ref{goebel}.\end{proof}

\begin{remark}
By \cite[Theorem~1.2]{KarimovRepovs}, the first integral homology group of the Harmonic Archipelago is also isomorphic to the group described in Theorem 4.2 above. Indeed, it can be calculated in much the same way as we computed the first integral homology group of the Griffiths twin cone. We leave the details to the reader.
\end{remark}

\section{The Hawaiian Earring and product properties}\label{hom}

\begin{proposition}\label{onto}
Suppose $A$ is either $\mathbb{Q}$, $\mathbb{Z}/p\mathbb{Z}$ or $\mathbb{J}_p$ for some prime $p$. Then there is a homomorphism $h: \pi_1(\mathbb{H},{\bf o})\rightarrow A$ with $h(\pi({\mathcal U},{\bf o}))=A$ for all ${\mathcal U}\in Cov(\mathbb{H})$.
\end{proposition}

\begin{proof}
First suppose $A=\mathbb{J}_p$. Consider the surjective homomorphism \[\tau:\prod_{m\in \mathbb{N}}\mathbb{Z}\rightarrow \mathbb{J}_p=\lim_{\leftarrow}(\mathbb{Z}/p\mathbb{Z}\leftarrow \mathbb{Z}/p^2\mathbb{Z}\leftarrow \mathbb{Z}/p^3\mathbb{Z}\leftarrow \cdots)\] given by $\tau(a_1, a_2, a_3, \dots)=([a_1],[a_1+pa_2],[a_1+pa_2+p^2a_3],\dots)$.
Since the direct sum $\bigoplus_{n\in\mathbb{N}} \prod_{m\in\mathbb{N}} \mathbb{Z}$ is a pure subgroup of the direct product $\prod_{n\in\mathbb{N}} \prod_{m\in\mathbb{N}} \mathbb{Z}$ and since $\mathbb{J}_p$ is pure-injective, we can extend $\bigoplus_{n\in \mathbb{N}}\tau: \bigoplus_{n\in\mathbb{N}} \prod_{m\in\mathbb{N}} \mathbb{Z} \rightarrow \mathbb{J}_p$ to a homomorphism $ \sigma:  \prod_{n\in\mathbb{N}} \prod_{m\in\mathbb{N}} \mathbb{Z}\rightarrow \mathbb{J}_p$.
 Choose a bijection $\rho:\mathbb{N}\times \mathbb{N}\rightarrow \mathbb{N}$ such that for every $k\in \mathbb{N}$ there is an $n\in \mathbb{N}$ with $\rho^{-1}(\{1,2,\dots,k\})\cap \left(\{n\}\times \mathbb{N}\right)=\emptyset$, e.g., the diagonal enumeration $\rho(n,m)= n+(n+m-1)(n+m-2)/2$. Consider the isomorphism $\varphi:\prod_{k\in \mathbb{N}}\mathbb{Z}\rightarrow \prod_{n\in\mathbb{N}}\prod_{m\in\mathbb{N}} \mathbb{Z}$, defined by $\varphi((a_k)_{k\in\mathbb{N}})=(a_{\rho(n,m)})_{(n,m)\in\mathbb{N}\times\mathbb{N}}$. Put $\phi=\sigma\circ \varphi:\prod_{k\in \mathbb{N}}\mathbb{Z} \rightarrow \mathbb{J}_p$. Then $\phi(\prod_{k=n}^\infty \mathbb{Z})=\mathbb{J}_p$ for all $n\in \mathbb{N}$. Define $r_k:\mathbb{H}\rightarrow C_k$  by $r_k(x)=x$ if $x\in C_k$ and $r_k(x)={\bf o}$ if $x\not\in C_k$. Let $\mu:\pi_1(\mathbb{H},{\bf o})\rightarrow \prod_{k\in \mathbb{N}} \mathbb{Z}$ be the composition of $(r_{n\#})_{n\in \mathbb{N}}:\pi_1(\mathbb{H},{\bf o})\rightarrow \prod_{k\in\mathbb{N}} \pi_1(C_k,{\bf o})$ and the canonical isomorphism $\prod_{k\in\mathbb{N}} \pi_1(C_k,{\bf o})\cong \prod_{k\in \mathbb{N}} \mathbb{Z}$.
Put $h=\phi\circ \mu: \pi_1(\mathbb{H},{\bf o})\rightarrow \mathbb{J}_p$.  Observe that $\mu(incl_\#(\pi_1(\bigcup_{k=n}^\infty C_k,{\bf o})))= \prod_{k=n}^\infty \mathbb{Z}$ for all $n\in \mathbb{N}$, because $\mu(\ell_1^{a_1}\ell_2^{a_2}\ell_3^{a_3}\cdots) =(a_1,a_2,a_3,\dots)$ for $a_i\in \mathbb{Z}$, where $\ell_i$ is as in the proof of Lemma~\ref{sort}, parametrized appropriately. Let $\mathcal U\in Cov(\mathbb{H})$. Choose $n\in \mathbb{N}$ with $incl_\#(\pi_1(\bigcup_{k=n}^\infty C_k,{\bf o}))\leqslant \pi({\mathcal U},{\bf o})$. Then $\mathbb{J}_p=\phi(\prod_{k=n}^\infty \mathbb{Z}) = \phi \circ \mu(incl_\#(\pi_1(\bigcup_{k=n}^\infty C_k, {\bf o})))\leqslant   \phi \circ \mu(\pi({\mathcal U},{\bf o}))=h(\pi({\mathcal U},{\bf o}))$.

This also covers the case $A=\mathbb{Z}/p\mathbb{Z}$, since there is an epimorphism $\mathbb{J}_p\rightarrow \mathbb{Z}/p\mathbb{Z}$.
If $A=\mathbb{Q}$, we instead start with $\tau:\bigoplus_{k\in \mathbb{N}}\mathbb{Z} \rightarrow \mathbb{Q}$, given by $\tau({\bf e}_k)=1/k$, where ${\bf e}_k(n)=0$ for $n\not=k$ and ${\bf e}_k(k)=1$, and extend it to $\phi:\prod_{k\in \mathbb{N}}\mathbb{Z}\rightarrow \mathbb{Q}$, noting that $\mathbb{Q}$ is injective. Then $\phi(\prod_{k=n}^\infty\mathbb{Z})=\mathbb{Q}$ for all $n\in \mathbb{N}$ and we can proceed as before.
\end{proof}

\begin{corollary}\label{HE}
Let $A$ be an abelian group which is homomorphically Hausdorff relative to the Hawaiian Earring $\mathbb{H}$. Then $A$ is cotorsion-free.
\end{corollary}

\begin{proof}
Combine Proposition~\ref{onto} with Theorem~\ref{goebel}.
\end{proof}

\begin{remark} Implicitly contained in our argument is also a proof of the following characterization, appearing in \cite[\S7]{CC1}, extracted by Dugas, G\"obel,  Wald et al.\@ from Nunke's characterization of slender groups \cite{Nunke}: An abelian group $A$ is cotorsion-free if and only if for every homomorphism $\phi:\mathbb{Z}^\mathbb{N}\rightarrow A$, we have $\bigcap_{n\in \mathbb{N}} \phi(\prod_{k=n}^\infty \mathbb{Z})=0$.
\end{remark}

While n-slenderness is preserved under restricted direct products (and free products) \cite[Theorem~3.6]{Eda1992}, the following holds for arbitrary direct products:

\begin{proposition}\label{prod}
A product $\prod_{i\in I} G_i$ is homomorphically Hausdorff or Spanier-trivial relative to $X$ if and only if each group $G_i$ has the corresponding property.
\end{proposition}

\begin{proof}
First suppose that each $G_i$ is homomorphically Hausdorff relative to $X$. Let $h:\pi_1(X,x)\rightarrow \prod_{i\in I} G_i$ be a homomorphism and consider the projections $p_j: \prod_{i\in I} G_i\rightarrow G_j$. Then $p_j(\bigcap_{{\mathcal U}\in Cov(X)} h(\pi({\mathcal U},x)))\leqslant \bigcap_{{\mathcal U}\in Cov(X)} p_j\circ h(\pi({\mathcal U},x))=1$ for every $j\in I$, so that $\bigcap_{{\mathcal U}\in Cov(X)} h(\pi({\mathcal U},x))\leqslant \bigcap_{j\in I} ker(p_j)=1$. The argument is analogous if each $G_i$ is Spanier-trivial relative to $X$. The converse is obvious.
\end{proof}

Let us call a group $G$ {\em residually n-slender} if for every $g\in G\setminus\{1\}$ there is a homomorphism $h:G\rightarrow S$ to an n-slender group $S$ such that $h(g)\not=1$. (See \cite{CE}.)

\begin{corollary}\label{residually}
Every residually n-slender group is homomorphically Hausdorff relative to every Peano continuum.
\end{corollary}

\begin{proof}
Since a group is residually n-slender if and only if it is isomorphic to a subgroup of a direct product of n-slender groups, this follows from Proposition~\ref{prod}.
\end{proof}

\begin{remark} Since every residually free group is, in particular, residually n-slender, Corollary~\ref{residually} applies to the following examples of fundamental groups:

If $X$ is a compact one-dimensional metric space or a proper compact subset of a surface, then $\pi_1(X,\ast)$ is isomorphic to a subgroup of a direct product of free groups of finite rank \cite{CF,FZ2005} and thus residually free.

 Suppose $Y=\prod_{n\in \mathbb{N}} Y_n$ is a product of countably many one-dimensional Peano continua $Y_n$, each of which is not semilocally simply-connected at any point. Then  $\pi_1(Y,\ast)\cong \prod_{n\in\mathbb{N}} \pi_1(Y_n,\ast)$ is a residually free group from whose isomorphism type one can recover the space $Y$  by a construction described in \cite{CE}. This construction, in turn, makes use of the fact that each $\pi_1(Y_n,\ast)$ is residually n-slender. Examples for $Y_n$ include the Menger curve, the Sierpi\'{n}ski curve, and the Sierpi\'{n}ski triangle.

Lastly, if $Z$ is a well-balanced tree of surfaces, then $\pi_1(Z,\ast)$ is isomorphic to a subgroup of a direct product of fundamental groups of closed surfaces \cite{FG} each of which is residually free. (Note that, with the exception of the three non-orientable surfaces of smallest genus, surface groups are fully residually free. See, for example, \cite[Lemma~5.5.11]{Chiswell}.) Hence, $\pi_1(Z,\ast)$ is residually free. Examples for $Z$ include the Pontryagin sphere (a densely iterated connected sum of tori) and the Pontryagin surface $\Pi_2$ (a densely iterated connected sum of real projective planes).
\end{remark}

\begin{remark} It has recently been shown that certain amalgamated free products and certain HNN extensions of n-slender groups are n-slender, among them the Baumslag-Solitar groups \cite{Nakamura}.
\end{remark}

\noindent {\bf Acknowledgements.} This research was partially supported by the Grant-in-Aid for Scientific Research (C) of Japan (No. 20540097 and 23540110 to Katsuya Eda) and by a grant from the Simons Foundation (No. 245042 to Hanspeter Fischer).


\begin{thebibliography}{00}

\bibitem{BZ} O. Bogopolski and A. Zastrow, {\em The word problem for some uncountable groups given by countable words}, Topology and its Application {\bf 159} (2012) 569--586.

\bibitem{BrazasFabel} J. Brazas and P. Fabel, {\em Thick Spanier groups and the first shape group}, Rocky Mountain Journal of Mathematics, to appear (arXiv:1207.1310).

\bibitem{BDLM} N. Brodskiy, J. Dydak, B. LaBuz and A. Mitra, {\em Covering maps for locally path-connected spaces},  Fundamenta Mathematicae {\bf 218} (2012) 13--46.

\bibitem{CC1} J.W. Cannon and G.R. Conner, {\em The combinatorial structure of the Hawaiian earring group}, Topology and its Applications {\bf 106} (2000) 225--271.


\bibitem{Chiswell}  I. Chiswell,  Introduction to $\Lambda$-trees. World Scientific Publishing, 2001.

\bibitem{CE} G.R. Conner and K. Eda, {\em  Fundamental groups having the whole information of spaces}, Topology and its Applications {\bf 146/147} (2005) 317--328; correction: ibid. {\bf 154} (2007) 771--773.


\bibitem{CF} M.L. Curtis and M.K. Fort, Jr., {\em  Singular homology of one-dimensional spaces}, Annals of Mathematics {\bf 69}
309--313 (1959).

\bibitem{DugasGoebel} M. Dugas and R. G\"obel, {\em Algebraisch kompakte Faktorgruppen},  Journal f\"ur die Reine und Angewandte Mathematik {\bf 307/308} (1979) 341--352.

\bibitem{Eda1991} K. Eda, {\em The first integral singular homology groups of one-point unions}, Quarterly Journal of Mathematics. Oxford Series (2) {\bf 42} (1991) 443--456.

\bibitem{Eda1992} K. Eda, {\em Free $\sigma$-products and noncommutatively slender groups}, Journal of Algebra {\bf 148} (1992) 243--263.

\bibitem{Eda1998} K. Eda,  {\em Free $\sigma$-products and fundamental groups of subspaces of the plane}, Topology and its Applications {\bf 84} (1998) 283--306.

\bibitem{Eda2005} K. Eda, {\em Algebraic topology of Peano continua}, Topology and its Applications {\bf 153} (2005) 213--226; correction: ibid. {\bf 154} (2007) 771--773.

\bibitem{Eda2010}  K. Eda, {\em Homotopy types of one-dimensional Peano continua},  Fundamenta Mathematicae {\bf 209} (2010) 27--42.

\bibitem{Eda2014} K. Eda, {\em Singular homology groups of one-dimensional spaces}, Preprint (\url{www.logic.info.waseda.ac.jp/~eda}).

\bibitem{EdaKawamura1998} K. Eda and K. Kawamura, {\em  The fundamental groups of one-dimensional spaces}, Topology and its Application {\bf 87} (1998) 163--172.

\bibitem{EdaKawamura2000} K. Eda and K. Kawamura, {\em The singular homology of the Hawaiian Earring}, Journal of the London Mathematical Society (2) {\bf 62} (2000) 305--310.


\bibitem{F} H. Fischer, {\em Arc-smooth generalized universal covering spaces}, Topology and its Applications {\bf 155} (2008) 1056--1065.

\bibitem{FG}  H. Fischer and C.R. Guilbault, {\em On the fundamental groups of trees of manifolds},  Pacific Journal of Mathematics {\bf 221} (2005) 49--79.

\bibitem{FZ2005} H. Fischer and A. Zastrow, {\em The fundamental groups of subsets of closed surfaces inject into their first shape groups}, Algebraic and  Geometric  Topology {\bf 5} (2005) 1655--1676.

\bibitem{FZ2007} H. Fischer and A. Zastrow, {\em Generalized universal covering spaces and the shape group}, Fundamenta Mathematicae {\bf 197} (2007) 167--196.

\bibitem{FZ2012}  H. Fischer and A. Zastrow, {\em Combinatorial $\mathbb{R}$-trees as generalized Cayley graphs
for fundamental groups of one-dimensional spaces}, Geometriae Dedicata {\bf 163} (2013) 19--43.

\bibitem{FZ2013} H. Fischer and A. Zastrow, {\em Word calculus in the fundamental group of the Menger curve}, Preprint (arXiv:1310.7968).

\bibitem{Fuchs} L. Fuchs, Infinite abelian groups, Vol. 1/2, Academic Press, 1970/1973.

\bibitem{GoebelWald} R. G\"obel and B. Wald, {\em  Wachstumstypen und schlanke Gruppen}, Symposia Mathematica {\bf 23} (1979)  201--239.

\bibitem{Griffiths1954} H.B. Griffiths, {\em The fundamental group of two spaces with a common point}, Quarterly Journal of Mathematics. Oxford Series (2) {\bf 5} (1954) 175--190; correction: ibid. {\bf 6} (1955) 154--155.

\bibitem{Griffiths} H.B. Griffiths, {\em  Infinite products of semi-groups and local connectivity}, Proceedings of the London Mathematical Society (3) {\bf 6} (1956) 455--480.

\bibitem{Higman} G. Higman, {\em Unrestricted free products, and varieties of topological groups}, Journal of the London Mathematical Society {\bf 27} (1952) 73--81.

\bibitem{KarimovRepovs} U.H. Karimov and D. Repov\v{s}, {\em On the homology of the harmonic archipelago}, Central European Journal of Mathematics {\bf 10} (2012) 863--872.

\bibitem{MorganMorrison} J.W. Morgan and I. Morrison, {\em A van Kampen theorem for weak joins}, Proceedings of the London Mathematical Society (3) {\bf 53} (1986) 562--576.

\bibitem{Nakamura} J. Nakamura, {\em Atomic properties of the Hawaiian Earring group for HNN extensions}, Communications in Algebra, to appear.


\bibitem{Nunke} R.J. Nunke, {\em Slender groups}, Bulletin of the American Mathematical Society {\bf 67} (1961) 274--275;  Acta Sci. Math. (Szeged) {\bf 23} (1962) 67--73.

\bibitem{Spanier} E.H. Spanier, Algebraic Topology, McGraw-Hill, 1966.

\bibitem{VZ} \v{Z}. Virk and A. Zastrow, {\em  The comparison of topologies related to various concepts of generalized covering spaces}, Topology and its   Applications {\bf 170} (2014) 52--62.

\end{thebibliography}
\end{document}